\documentclass[11pt]{amsart}

\usepackage{amsfonts, amstext, amsmath, amsthm, amscd, amssymb} 
\usepackage{mathtools}

\usepackage{float} 
\usepackage{graphics}
\usepackage{caption}
\usepackage{subcaption}
\usepackage{import} 

\usepackage{physics}  
\usepackage{microtype}

\usepackage[hidelinks,pagebackref,pdftex]{hyperref}

\usepackage[dvipsnames]{xcolor}
\definecolor{MyCyan}{HTML}{00F9DE}
\usepackage{enumitem} 
\setlist{nolistsep,leftmargin=*}
\usepackage{transparent} 
\usepackage{mathrsfs}

\usepackage{marginnote}
\long\def\@savemarbox#1#2{\global\setbox#1\vtop{\hsize\marginparwidth 
  \@parboxrestore\tiny\raggedright #2}}
\renewcommand{\setminus}{{\smallsetminus}}

\newtheorem{theorem}{Theorem}[section]
\newtheorem*{theorem*}{Theorem}
\newtheorem{proposition}[theorem]{Proposition}
\newtheorem{lemma}[theorem]{Lemma}
\newtheorem{corollary}[theorem]{Corollary}

\newtheorem*{namedtheorem}{\theoremname}
\newcommand{\theoremname}{testing}

\theoremstyle{definition}
\newtheorem{definition}[theorem]{Definition}

\newcommand{\refthm}[1]{Theorem~\ref{Thm:#1}}

\newcommand{\reffig}[1]{Figure~\ref{Fig:#1}}

\title{Generalised T-links Represent All Links in the 3-Sphere}

\author{Thiago de Paiva}
\address[]{Beijing International Center for Mathematical Research, Peking University, Beijing 100871, China P.R.}
\email[]{thhiagodepaiva@gmail.com}

\author[Connie Hui]{Connie On Yu Hui} 
\address[]{Sydney Mathematical Research Institute (SMRI), The University of Sydney, NSW 2006, Australia } 
\email[]{connieonyuhui@gmail.com} 

\author{Jos\'{e} Andr\'{e}s Rodr\'{\i}guez Migueles}
\address[]{Centro de Investigaci\'on en Matem\'aticas, GTO 36023, Mexico}
\email[] {jose.migueles@cimat.mx} 

\begin{document} 

\begin{abstract}
We introduce generalised T-links, a restricted extension of the widely
studied families of T-links and twisted torus links, and prove that this
family is universal: every link in the \(3\)-sphere can be represented as
a generalised T-link.

This gives a compact braid-theoretic form of Ghrist's universality
theorem for certain Lorenz-like templates. More precisely, we replace the
template-theoretic representatives coming from universal Lorenz-like
templates with explicit braid closures depending on finitely many integer
parameters. Although closely related braid families have appeared
previously in the literature, the universality of this restricted class is
the main point of the present paper. Thus, all links in the \(3\)-sphere
are placed inside a natural structured class closely related to T-links,
Lorenz links, and twisted torus links.
\end{abstract}

\maketitle

\section{Introduction}

The study of links via braid diagrams is an important theme in knot theory. By Alexander's theorem, every link in \(S^3\) can be
represented as the closure of a braid. However, an arbitrary braid
representative usually carries very little structure, and many problems in
knot theory become more tractable only when the link is represented by a
braid belonging to a special family. Important examples include positive
braids, closed \(3\)-braids, torus links, twisted torus links, Lorenz
links, and T-links. These families have been studied from several
perspectives, including braid index
\cite{de2024lorenz, Franks-Williams:BraidsAndJonesPolynomial, FellerHubbard}, knot Floer homology
\cite{homology}, geometric classification
\cite{Composite, lee2018satellite, Anote, tangle, dePaiva2025,
de2022torus, dePaivaPurcell2024, twofulltwists, de2022hyperbolic},
hyperbolic volume \cite{generalizedtwistedtoruslinks, nextsimplest,
simplest}, essential surfaces \cite{MR2788641}, Heegaard splittings
\cite{Heegaard}, and bridge spectra \cite{Bridge}.

Let \(B_p\) denote the braid group on \(p\) strands, generated by the
standard Artin generators \(\sigma_1,\dots,\sigma_{p-1}\). For integers
\(2\leq r_1<\cdots<r_k\) and \(s_1,\dots,s_k>0\), the T-link
\[
T((r_1,s_1),\dots,(r_k,s_k))
\]
is defined to be the closure of the positive braid
\[
(\sigma_1\sigma_2\cdots\sigma_{r_1-1})^{s_1}
(\sigma_1\sigma_2\cdots\sigma_{r_2-1})^{s_2}
\cdots
(\sigma_1\sigma_2\cdots\sigma_{r_k-1})^{s_k}
\in B_{r_k}.
\]

Among these families, Lorenz links and T-links occupy a particularly
interesting position because they connect low-dimensional topology with
dynamics. Lorenz links arise as periodic orbits of the Lorenz flow, or,
equivalently, as links embedded in the Lorenz template, a branched surface
equipped with a semiflow that models the dynamics of the Lorenz attractor.
Birman and Williams introduced templates as a general framework for
studying knots and links arising as periodic orbits of flows in
\(3\)-manifolds. In this setting, Lorenz links provide one of the clearest
bridges between symbolic dynamics, braid theory, and knot theory.

A theorem of Birman and Kofman shows that a broader definition of Lorenz links, which includes any parallel push-offs on the
Lorenz template of any sublinks of the previous notion of Lorenz links, admit a remarkably
simple braid-theoretic description: they are precisely T-links
\cite{Birman-Kofman:NewTwistOnLorenzLinks}. Thus, a class of links
originally defined from dynamics can be described by explicit positive
braid closures depending on finitely many integer parameters. This
correspondence makes braid-theoretic methods available in the study of
Lorenz links. It replaces a template-theoretic description with explicit
braid words and places Lorenz links inside a structured class that also
contains torus links and is closely related to twisted torus links.

Twisted torus links provide another natural enlargement of the class of
torus links. Twisted torus knots were introduced by Dean in his doctoral
thesis in connection with the study of Dehn surgeries producing
Seifert-fibered spaces \cite{Thesis}. They are obtained from a torus link
\(T(p,q)\) by adding \(s\) full twists on \(r\) adjacent strands. Here the
twisting parameter \(s\) is allowed to be positive or negative. In
particular, when \(r\leq p\), the twisted torus link \(T(p,q;r,s)\) is
represented by the braid
\[
(\sigma_1\sigma_2\cdots\sigma_{r-1})^{sr}
(\sigma_1\sigma_2\cdots\sigma_{p-1})^{q}.
\]
This family appears in several of the works cited above, where it is
studied from different perspectives, including geometric classification,
hyperbolic volume, essential surfaces, and knot Floer homology; see, for
example, \cite{unexpected, LeeThiago, de2022hyperbolic} for more
information.

From the braid-theoretic point of view, twisted torus links are closely
related to T-links, but the two families have different kinds of
flexibility. A twisted torus link has two twisting regions: the original
torus-link region and the additional twisting region on \(r\) adjacent
strands, where the additional twisting may be positive or negative. By
contrast, T-links are closures of positive braids, but they allow
arbitrarily many positive twisting regions. Generalised T-links combine
these two features in a restricted way: they allow several positive
twisting regions, as in T-links, together with one final twisting parameter
that may be negative , as in twisted torus links.

Torus links occur as basic examples both among T-links and among twisted
torus links. By the Birman--Kofman theorem, Lorenz links coincide with
T-links. Twisted torus links form a different enlargement of torus links:
they allow one additional twisting region, possibly with negative twisting.

Neither T-links nor twisted torus links form a universal family. Indeed,
twisted torus knots have very restricted behaviour when they are
composite: Lee classified the composite twisted torus knots and showed
that, in that case, they are connected sums of two torus knots
\cite{lee2018satellite}. Thus, for example, composite knots with more
than two prime factors cannot all be realised as twisted torus knots.
Similarly, T-links, equivalently Lorenz links, do not contain all satellite
knots. In particular, there are infinitely many satellite knots whose
companions and patterns are Lorenz knots but which are not themselves
Lorenz knots \cite{dePaiva2025}. Therefore, although T-links, Lorenz
links, and twisted torus links are natural and widely studied families,
they do not contain all links in \(S^3\).

This motivates the following common generalisation of T-links and twisted
torus links. Let
\(2\leq r_1<\cdots<r_k<r_{k+1}\), let
\(s_1,\dots,s_k>0\), let \(\vert s_{k+1}\vert \geq 2\), and let \(d\geq 0\). We define
the \textit{generalised T-link}
\[
T((r_1,s_1),\dots,(r_k,s_k),(r_{k+1},s_{k+1}),d)
\]
to be the closure of the braid
\[
(\sigma_{d+1}\cdots\sigma_{d+r_1-1})^{s_1}
\cdots
(\sigma_{d+1}\cdots\sigma_{d+r_k-1})^{s_k}
(\sigma_{d+1}\cdots\sigma_{d+r_{k+1}-1})^{s_{k+1}}.
\]
Here the final twisting parameter \(s_{k+1}\) is allowed to be negative,
while all preceding twisting parameters are positive.

Our main result is that this restricted generalisation of T-links and
twisted torus links is universal.

\begin{theorem}\label{Thm:AllLinksGeneralisedTLinks}
Every link in \(S^3\) can be represented as a generalised T-link.
\end{theorem}

Some variations of generalised T-links have appeared previously in the
knot theory literature, although their universality does not seem to have
been observed. For instance, closely related families arise in the study
of volume bounds for generalized twisted torus links
\cite{generalizedtwistedtoruslinks} and in work relating braid orderings
to the geometry of closed braids \cite{MR2788641}. In those settings, one
usually allows more general patterns of twisting. The point of the present
paper is that a much more restricted family is already universal: it is
enough to allow several positive twisting regions, together with one final
twisting parameter that may be negative and an offset in the braid
support.

Theorem~\ref{Thm:AllLinksGeneralisedTLinks} can be viewed as a
braid-theoretic refinement of Ghrist's universality theorem for
Lorenz-like templates. The notions of template and universal template are
recalled in Section~\ref{Sec:Templates}. Ghrist proved that certain
Lorenz-like templates are universal, in the sense that they contain
representatives of all links in \(S^3\), up to ambient isotopy. The
contribution here is to convert these template representatives into
explicit braid closures belonging to a natural family closely related to
T-links, Lorenz links, and twisted torus links. 

This conversion does more than providing another finite description of the
same links. It replaces template representatives by braid words built from
a small number of repeated blocks in the standard Artin generators. In
many cases, the isotopies and destabilizations used in the conversion also
lead to braid representatives with fewer crossings. Thus, the
generalised T-link form gives a concrete braid-theoretic parametrisation
of the links arising from universal Lorenz-like templates.

This is useful because generalised T-links naturally enlarge families that
have already been studied extensively. Several subfamilies are already
covered by results on braid index
\cite{de2024lorenz, Franks-Williams:BraidsAndJonesPolynomial, FellerHubbard},
knot Floer homology \cite{homology}, geometric classification
\cite{Composite, lee2018satellite, Anote, tangle, dePaiva2025,
de2022torus, dePaivaPurcell2024, twofulltwists, de2022hyperbolic},
hyperbolic volume \cite{generalizedtwistedtoruslinks, nextsimplest,
simplest}, essential surfaces \cite{MR2788641}, Heegaard splittings
\cite{Heegaard}, and bridge spectra \cite{Bridge}. In this way, the
generalised T-link description places a universal family of links inside a
framework where a substantial amount of braid-theoretic, geometric, and
topological machinery is already available.

This braid-theoretic viewpoint also makes the non-uniqueness of universal
template representations more explicit. A fixed link may be represented
in many different ways inside universal Lorenz-like templates, and the
conversion above turns such representatives into explicit generalised
T-link braid words. This does not solve the uniqueness problem for
generalised T-link parameters: a single link may admit several, and in
some cases infinitely many, generalised T-link descriptions. Rather, it
provides a braid-theoretic setting in which this redundancy can be studied.

This phenomenon is already visible for ordinary T-links. We prove that
every T-link admits infinitely many representatives in each universal even
Lorenz-like template \(\mathscr L(0,2n)\), with \(n<0\),
Theorem~\ref{Thm:InfiniteTripNumber}. Thus, even within the classical
class of T-links, universal Lorenz-like templates contain substantial
redundancy.

The proof of \refthm{AllLinksGeneralisedTLinks} proceeds in two steps. First, we show that the ambient isotopy
classes of links carried by the Lorenz-like template
\(\mathscr L(0,2v)\), with \(v\leq 0\), are precisely the ambient isotopy
classes represented by generalised T-links satisfying
$s_{k+1}\geq v r_{k+1}.$
In particular, this extends the Birman--Kofman correspondence between
Lorenz links and T-links from the Lorenz template to a family of
Lorenz-like templates. Second, when \(v<0\), Ghrist's theorem implies that
\(\mathscr L(0,2v)\) is universal. Combining these two facts proves
Theorem~\ref{Thm:AllLinksGeneralisedTLinks}.

The paper is organized as follows. In Section~\ref{Sec:Templates}, we
recall the required background on Lorenz templates, Lorenz-like templates,
and universal templates. In Section~\ref{ANewTwist}, we prove the
conversion from Lorenz-like template representatives to generalised
T-link representatives and deduce that every link in \(S^3\) is a
generalised T-link. In Section~\ref{Sec:FinalRemarks}, we discuss some
questions arising from this representation theorem and indicate directions
for further work.

\section{Lorenz-like templates and trip number}
\label{Sec:Templates}

A \emph{template} is a compact branched \(2\)-manifold with nonempty
boundary, equipped with a semiflow locally modelled on joining and
splitting charts; see
\cite[Figure~2.4]{Ghrist-Holmes-Sullivan:KnotsALinksInThreeDimFlows}.
A fundamental example is the \emph{Lorenz template}
(see Figure~\ref{Fig:LorenzTemplate}, left). It was introduced by
Guckenheimer and Williams
\cite{Guckenheimer-Williams:StructuralStabilityOfLorenzAttractors,
Williams:StructureOfLorenzAttractors}
as a geometric model for the Lorenz flow, and Tucker
\cite{Tucker:Smale14thProblem} later justified this model. We regard a
\emph{Lorenz link} as a link embedded in the Lorenz template; equivalently,
Lorenz links are the links represented by periodic orbits carried by this
template.  

\begin{definition}\label{Def:LorenzLikeTemplate}
Let \(u\) and \(v\) be integers. The \emph{\((u,v)\)-Lorenz-like template}
\(\mathscr L(u,v)\) is obtained by cutting along
the branch line of the Lorenz template, regluing the left strip with \(u\) half twists, and
regluing the right strip with \(v\) half twists. Positive and negative
half twists are distinguished by the sign of the crossings.

A link embedded in \(\mathscr L(u,v)\) is called a
\emph{\((u,v)\)-Lorenz-like link}, or simply a \emph{Lorenz-like link}.
\end{definition}

\begin{figure}
%% Creator: Inkscape 1.1.2 (b8e25be833, 2022-02-05), www.inkscape.org
%% Accompanies image file 'LorenzLikeTemplate.pdf' (pdf, eps, ps)
\begingroup%
  \makeatletter%
  \providecommand\color[2][]{%
    \errmessage{(Inkscape) Color is used for the text in Inkscape, but the package 'color.sty' is not loaded}%
    \renewcommand\color[2][]{}%
  }%
  \providecommand\transparent[1]{%
    \errmessage{(Inkscape) Transparency is used (non-zero) for the text in Inkscape, but the package 'transparent.sty' is not loaded}%
    \renewcommand\transparent[1]{}%
  }%
  \providecommand\rotatebox[2]{#2}%
  \newcommand*\fsize{\dimexpr\f@size pt\relax}%
  \newcommand*\lineheight[1]{\fontsize{\fsize}{#1\fsize}\selectfont}%
  \ifx\svgwidth\undefined%
    \setlength{\unitlength}{144.61663982bp}%
    \ifx\svgscale\undefined%
      \relax%
    \else%
      \setlength{\unitlength}{\unitlength * \real{\svgscale}}%
    \fi%
  \else%
    \setlength{\unitlength}{\svgwidth}%
  \fi%
  \global\let\svgwidth\undefined%
  \global\let\svgscale\undefined%
  \makeatother%
  \begin{picture}(1,0.7315248)%
    \lineheight{1}%
    \setlength\tabcolsep{0pt}%
    \put(0,0){\includegraphics[width=\unitlength,page=1]{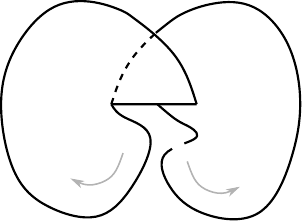}}%
    \put(0.3075093,0.06912099){\color[rgb]{0.6,0.6,0.6}\makebox(0,0)[lt]{\lineheight{1.25}\smash{\begin{tabular}[t]{l}$x$\end{tabular}}}}%
    \put(0.67403551,0.04283614){\color[rgb]{0.6,0.6,0.6}\makebox(0,0)[lt]{\lineheight{1.25}\smash{\begin{tabular}[t]{l}$y$\end{tabular}}}}%
    \put(0,0){\includegraphics[width=\unitlength,page=2]{LorenzLikeTemplate.pdf}}%
  \end{picture}%
\endgroup%

\caption{The Lorenz-like template \(\mathscr L(1,2)\).}
\label{Fig:LorenzLikeTemplate}
\end{figure}

The templates \(\mathscr L(u,v)\) and \(\mathscr L(v,u)\) carry the same
family of links, up to ambient isotopy.

\begin{definition}\label{Def:TripNumber}
Let \(L\) be a link embedded in a Lorenz-like template. 
%The \emph{trip number} of \(L\), denoted by \(\tau(L)\), is the number of strands of \(L\) which, after the template is cut open along its branch line, travel from the left strip to the right strip. 
The
\emph{trip number} of \(L\), denoted by \(\tau(L)\), is the number of strands in the left strip of the braid Lorenz template (see Figure~\ref{Fig:LorenzTemplate}, right) that intersect the part of the branch line which connects the right strip. 
\end{definition}

%For links in the Lorenz template \(\mathscr L(0,0)\), this agrees with
%the trip number introduced by Birman and Williams. Franks and Williams
%proved that the trip number of a Lorenz link is equal to its braid index.
%If \(L\) is a Lorenz-like link in \(\mathscr L(0,n)\), then \(L\) is
%obtained from an associated Lorenz link by inserting \(n\) half twists in
%the right strip. These half twists do not change the trip number.

\begin{definition}\label{Def:UniversalTemplate}
A template \(\mathcal T\) embedded in \(S^3\) is called
\emph{universal} if every link in \(S^3\) is ambient isotopic to a link
embedded in \(\mathcal T\). Equivalently, for every ambient isotopy class
of links in \(S^3\), the template \(\mathcal T\) carries at least one
representative of that class.
\end{definition}

\begin{proposition}[
{\cite{Ghrist:BranchedTwoMfdsSupportingAllLinks}} and
{\cite[Proposition~3.2.16]{Ghrist-Holmes-Sullivan:KnotsALinksInThreeDimFlows}}
]\label{Prop:UniversalTemplates}
If \(v<0\), then the Lorenz-like template \(\mathscr L(0,v)\) is universal.
\end{proposition}

%In particular, every link in \(S^3\) can be embedded in the Lorenz-like
%template \(\mathscr L(0,-2)\).

\section{A braid simplification of Lorenz-like links}
\label{ANewTwist} 

In this section we prove the braid-theoretic conversion underlying the
main theorem. The starting point is the Birman--Kofman correspondence,
which identifies Lorenz links with T-links
\cite{Birman-Kofman:NewTwistOnLorenzLinks}. We extend this correspondence
from the Lorenz template \(\mathscr L(0,0)\) to the Lorenz-like templates
\(\mathscr L(0,2v)\). More precisely, we show that, for each \(v\leq 0\), the ambient isotopy classes of links carried by \(\mathscr L(0,2v)\) are
precisely the ambient isotopy classes represented by generalised T-links
satisfying a simple inequality on the final twisting parameter.

Although this result is inspired by the Birman--Kofman theorem, the
conversion used here is different. The additional twists in the
Lorenz-like templates must be separated from the positive part of the
Lorenz braid and then absorbed into the final twisting block of a
generalised T-link. This is the mechanism that turns a Lorenz-like
template representative into a compact braid word of generalised T-link
type.

When \(v<0\), Ghrist's theorem implies that the template
\(\mathscr L(0,2v)\) is universal. Thus the conversion theorem immediately
implies that every link in \(S^3\) can be represented as a generalised
T-link. We also record a related non-uniqueness phenomenon: every ordinary
T-link has infinitely many representatives in each universal
Lorenz-like template of the form \(\mathscr L(0,2n)\), with \(n<0\).

\begin{figure}
% left:
%% Creator: Inkscape 1.1.2 (b8e25be833, 2022-02-05), www.inkscape.org
%% Accompanies image file 'LorenzTemplate.pdf' (pdf, eps, ps)
\begingroup%
  \makeatletter%
  \providecommand\color[2][]{%
    \errmessage{(Inkscape) Color is used for the text in Inkscape, but the package 'color.sty' is not loaded}%
    \renewcommand\color[2][]{}%
  }%
  \providecommand\transparent[1]{%
    \errmessage{(Inkscape) Transparency is used (non-zero) for the text in Inkscape, but the package 'transparent.sty' is not loaded}%
    \renewcommand\transparent[1]{}%
  }%
  \providecommand\rotatebox[2]{#2}%
  \newcommand*\fsize{\dimexpr\f@size pt\relax}%
  \newcommand*\lineheight[1]{\fontsize{\fsize}{#1\fsize}\selectfont}%
  \ifx\svgwidth\undefined%
    \setlength{\unitlength}{141.44883994bp}%
    \ifx\svgscale\undefined%
      \relax%
    \else%
      \setlength{\unitlength}{\unitlength * \real{\svgscale}}%
    \fi%
  \else%
    \setlength{\unitlength}{\svgwidth}%
  \fi%
  \global\let\svgwidth\undefined%
  \global\let\svgscale\undefined%
  \makeatother%
  \begin{picture}(1,0.46887706)%
    \lineheight{1}%
    \setlength\tabcolsep{0pt}%
    \put(0,0){\includegraphics[width=\unitlength,page=1]{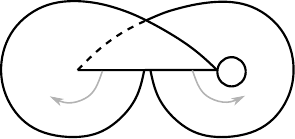}}%
    \put(0.24443091,0.07011143){\color[rgb]{0.6,0.6,0.6}\makebox(0,0)[lt]{\lineheight{1.25}\smash{\begin{tabular}[t]{l}$x$\end{tabular}}}}%
    \put(0.68548454,0.07368241){\color[rgb]{0.6,0.6,0.6}\makebox(0,0)[lt]{\lineheight{1.25}\smash{\begin{tabular}[t]{l}$y$\end{tabular}}}}%
    \put(0,0){\includegraphics[width=\unitlength,page=2]{LorenzTemplate.pdf}}%
  \end{picture}%
\endgroup%

\hspace{16mm} 
% right: 
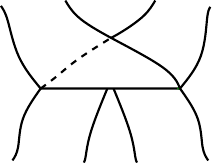

\caption{Left: Lorenz template. Right: Braid Lorenz template.
\label{Fig:LorenzTemplate}}
\end{figure}

The Lorenz template \( \mathscr{L}(0, 0)\) produces a braid template for the Lorenz links. This is achieved by cutting and opening the Lorenz template, as illustrated in Figure~\ref{Fig:LorenzTemplate}.

A Lorenz braid in the braid template is formed by groups of overcrossing and undercrossing strands, with the overcrossing strands located among the leftmost strands and the undercrossing strands among the rightmost strands at the top of the Lorenz braid. Furthermore, the overcrossing strands never cross each other, and the undercrossing strands never cross each other. See \reffig{LorenzBraid} for example. 
% A Lorenz braid is characterized by the permutation associated with its strands. \marginconnie{For the last sentence here, permutation of what? Is it referring to the full bunch order in the bunch algorithm? Shall we either remove this sentence and add ``Refer to \reffig{LorenzBraid}''; or, clearify and add a reference for this statement?\vspace{10pt}\\ I just saw the paragraph after the next one that explains ``permutation''. I think ``positions of the overcrossing strand bottom end points'' is not a permutation?  I will try to adjust the wording there and come back to this later.}
 
If a Lorenz braid has no overcrossing strands, or equivalently, if the Lorenz link has a trip number equal to zero, then the Lorenz link is the unlink with trivial components that are parallel to the two boundary components of the Lorenz template. Therefore, we consider that the Lorenz braid has at least one overcrossing strand. Furthermore, given a link \( L \) in the Lorenz template, we can obtain infinitely many splittable links by adding trivial loops that are parallel to the left boundary component of the template \( \mathscr{L}(0, 0) \). For now, when we consider a link in the Lorenz template, we disregard these trivial loops that are parallel to the left boundary component.

A Lorenz braid can be defined by determining the positions of the bottom end points of all the overcrossing strands. This is well-defined because once the positions of the overcrossing strands are determined, there is only one way to place the undercrossing strands since they do not cross one another.
% A Lorenz braid is well-defined by the permutation associated with the overcrossing strands because once the positions of the overcrossing strands are set, there is only one way to place the undercrossing strands since they do not cross. 

Let $p$ be the number of overcrossing strands. Label the top end points of all strands from left to right by $1, 2, 3, \ldots, p, \ldots$. (See \reffig{LorenzBraid}.) Vertically below each top label, mark the same number as a bottom label. For any integer $i$ from $1$ to $p$, let $r_i$ be a positive integer such that an overcrossing strand with top end point located at Top Label $i$ will have the bottom end point located at Bottom Label $(i+r_i)$. Since overcrossing strands never intersect, we obtain a non-decreasing sequence of natural numbers:
%The overcrossing strands are inclined from the left to the right. If an overcrossing strand begins at \( i \), then it ends at \( (i + r_i) \) with \( r_i > 0 \). This provides the following sequence of natural numbers:
\[
1 \leq r_1 \leq r_2 \leq \dots \leq r_{p-1} \leq r_p.
\]
We can assume that \( r_1 > 1 \) because if \( r_1 = 1 \), we can apply a destabilization move on the left of this Lorenz braid of \( L \) to obtain a new Lorenz braid for \( L \) with \( p - 1 \) overcrossing strands. Hence, we can assume that \( r_1 \geq 2 \).
Furthermore, since it is possible that \( r_i = r_{i+1} \), if we collect the strands of the same slope together, we can rewrite this sequence of natural numbers as the following vector, called a \textit{Lorenz vector},
\[
\vec{d}_{L} = \langle d_1^{s_1}, \dots, d_k^{s_k} \rangle,
\]
where \( 2 \leq d_1 < \dots < d_k \), \(\{ r_1, r_2, \dots, r_p \} = \{ d_1, \dots, d_k \}\), and \( s_1 + \dots + s_k = p \).

\begin{figure}
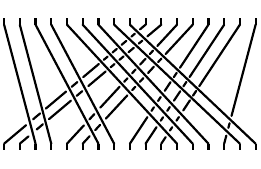
\caption{The Lorenz braid of the link $L$ represented by the Lorenz vector $\vec{d}_{L} = \langle 2^{2}, 4^{2}, 7^{3}, 8^{2}\rangle$.
\label{Fig:LorenzBraid}} 
\end{figure}  

Let $\vec{d}_{L} = \langle d_1^{s_1}, \dots, d_k^{s_k}\rangle$ be a Lorenz vector representing the Lorenz braid shown in Figure~\ref{Fig:LorenzBraid}. 
Denote by $(d_1, s_1)$ the braid 
$$(\sigma_{s_1}\dots \sigma_{d_1+s_1-1})\dots (\sigma_{2}\dots \sigma_{d_1+1})(\sigma_{1}\dots \sigma_{d_1}).$$
% $$(\sigma_{1}\dots \sigma_{d_1})(\sigma_{2}\dots \sigma_{d_1+1})\dots (\sigma_{s_1}\dots \sigma_{d_1+s_1-1}).$$

For integer $i>1$, denote by $(d_i, s_i)$ the braid
$$(\sigma_{s_1+\dots + s_{i-1}+ s_i}\dots \sigma_{s_1+\dots + s_{i-1}+s_i+d_i-1})\dots(\sigma_{s_1+\dots + s_{i-1}+1}\dots \sigma_{s_1+\dots + s_{i-1}+d_i}).$$ 
% $$(\sigma_{s_1+\dots + s_{i-1}+1}\dots \sigma_{s_1+\dots + s_{i-1}+d_i}) \dots 
% (\sigma_{s_1+\dots + s_{i-1}+ s_i}\dots \sigma_{s_1+\dots + s_{i-1}+s_i+d_i-1}).$$

The Lorenz braid associated with the Lorenz vector $\vec{d}_{L} = \langle d_1^{s_1}, \dots, d_k^{s_k}\rangle$ is the following braid
$$(d_k, s_k)(d_{k-1}, s_{k-1}) \dots (d_2, s_2)(d_1, s_1),$$
which has $s_1+\dots + s_k + d_k$ strands.

Let $a\geq b>1$. We denote by $\Delta_{b, a}$ the braid
$$(\sigma_{a-b+1}\sigma_{a-b+2}\dots \sigma_{a - 1})(\sigma_{a-b+1}\sigma_{a-b+2}\dots \sigma_{a - 2})\dots (\sigma_{a-b+1}\sigma_{a-b+2})(\sigma_{a-b+1}).$$
In other words, $\Delta_{b, a}$ denotes a positive half twist on the last $b$ strands of the trivial braid with $a$ strands.
On the other hand, we denote by $\Delta_{b, a}^{-1}$ a negative half twist on the last $b$ strands of the trivial braid with $a$ strands, as shown in the following braid
$$(\sigma_{a - 1}^{-1}\sigma_{a - 2}^{-1} \dots \sigma_{a-b+2}^{-1}\sigma_{a-b+1}^{-1})(\sigma_{a - 1}^{-1}\sigma_{a - 2}^{-1} \dots \sigma_{a-b+2}^{-1})\dots (\sigma_{a - 1}^{-1}\sigma_{a - 2}^{-1})(\sigma_{a - 1}^{-1}).$$ 

Every link \( L \) in the Lorenz-like template \( \mathscr{L}(0,n) \) is obtained from a Lorenz link \( L' \), called the \textit{Lorenz link associated with \( L \)}, by applying \( n \) half twists on the strands of \( L' \) in the right strip of the Lorenz template \( \mathscr{L}(0,0) \). If the Lorenz braid of \( L' \) has no overcrossing strands, or equivalently, if the Lorenz link \( L' \) has a trip number equal to zero, then \( L \) is a collection of trivial loops that are parallel to the two boundary components of the template with \( n \) half twists applied to some of them.

Furthermore, given a link \( L \) in the Lorenz-like template \( \mathscr{L}(0,n) \), we can obtain infinitely many splittable  links by adding trivial loops that are parallel to the left boundary component of the template \( \mathscr{L}(0,n) \). From now on, unless otherwise specified, when we consider a link in the Lorenz-like template \( \mathscr{L}(0,n) \), we disregard the trivial loops that are parallel to the left boundary component.

Moreover, if the Lorenz braid of \( L' \) has only one overcrossing strand, or equivalently, if the Lorenz link \( L' \) has a trip number equal to one, and there is only one strand in the right strip of the template \( \mathscr{L}(0,0) \), then \( L \) and \( L' \) are both the same unlink. We consider then that there is at least one overcrossing strand and at least two strands in the right strip of the template \( \mathscr{L}(0,0) \) in the Lorenz braid of \( L' \).

Let \(\vec{d}_{L'} = \langle d_1^{s_1}, \dots, d_k^{s_k} \rangle\) be the Lorenz vector of \(L'\), called the \textit{Lorenz vector associated with the link \(L\)}. Then, there exists a natural number \(d \geq 0\) such that \(L\) is the closure of the following braid with \(s_1 + \dots + s_k + d_k + d\) strands
\[
(d_k, s_k)(d_{k-1}, s_{k-1}) \dots (d_2, s_2)(d_1, s_1) \Delta^n_{d_k + d, s_1 + \dots + s_k + d_k + d.}
\]
The number \(d\) can be greater than zero if we add trivial loops that are parallel to the boundary component of the strip of the Lorenz template that is twisted to obtain the Lorenz-like template \(\mathscr{L}(0,n)\).  

%We observe that the trip number of \(L\) in the Lorenz-like template \(\mathscr{L}(0,n)\) is equal to the trip number of \(L'\) in the Lorenz template \(\mathscr{L}(0,0)\) because the \(n\) half twists applied along the strands of \(L'\) in the right strip of \(\mathscr{L}(0,n)\) do not change the trip number.

\begin{proposition}\label{Proposition1}
Let $\vec{d}_{L'} = \langle d_1^{s_1}, \dots, d_k^{s_k}\rangle$ be a Lorenz vector with $k>1$ or if $k = 1$ then $s_1>1$. 
Consider $\beta =  \sigma_{i_1}\dots \sigma_{i_a}$ a braid with at least $s_1+\dots + s_k + d_k$ strands and  no crossings $\sigma_{i_j}$ with $i_j\leq s_1+\dots + s_k$. Then, there is an ambient isotopy that takes the closure of the braid $$B'=(d_k, s_k)(d_{k-1}, s_{k-1}) \dots (d_2, s_2)(d_1, s_1)\beta$$ to the closure of the braid 
$$B''= (d_k, s_k)\dots (d_2, s_2)(d_1, s_1-1)\beta'(\sigma_{s_1+\dots + s_k}\dots \sigma_{s_1+\dots + s_k+d_1-2})$$ 
with $\beta' =  \sigma_{i_1-1}\dots \sigma_{i_a-1}$, where $B''$ has one fewer strand than $B'$, and the subscripts of $\sigma$'s in each $(d_i, s_i)$ change accordingly with the number of strands. 
\end{proposition} 

\begin{proof} 
The $(d_1+1)$-th strand at the bottom of the sub-braid $(d_k, s_k) \dots (d_1, s_1-1)$
anticlockwise goes around the braid closure once to end in the $(1)$-st strand at the bottom.
Thus, the $(d_1+1)$-th strand is connected to the $(1)$-st strand by an over strand that goes one time around the braid closure as illustrated by the red strand in the second drawing of Figure~\ref{Fig:ReduceOneStrand}.  

\begin{figure}
% leftmost:
%% Creator: Inkscape 1.1.2 (b8e25be833, 2022-02-05), www.inkscape.org
%% Accompanies image file 'ReduceOneStrand1.pdf' (pdf, eps, ps)
\begingroup%
  \makeatletter%
  \providecommand\color[2][]{%
    \errmessage{(Inkscape) Color is used for the text in Inkscape, but the package 'color.sty' is not loaded}%
    \renewcommand\color[2][]{}%
  }%
  \providecommand\transparent[1]{%
    \errmessage{(Inkscape) Transparency is used (non-zero) for the text in Inkscape, but the package 'transparent.sty' is not loaded}%
    \renewcommand\transparent[1]{}%
  }%
  \providecommand\rotatebox[2]{#2}%
  \newcommand*\fsize{\dimexpr\f@size pt\relax}%
  \newcommand*\lineheight[1]{\fontsize{\fsize}{#1\fsize}\selectfont}%
  \ifx\svgwidth\undefined%
    \setlength{\unitlength}{71.37316954bp}%
    \ifx\svgscale\undefined%
      \relax%
    \else%
      \setlength{\unitlength}{\unitlength * \real{\svgscale}}%
    \fi%
  \else%
    \setlength{\unitlength}{\svgwidth}%
  \fi%
  \global\let\svgwidth\undefined%
  \global\let\svgscale\undefined%
  \makeatother%
  \begin{picture}(1,1.18112865)%
    \lineheight{1}%
    \setlength\tabcolsep{0pt}%
    \put(0,0){\includegraphics[width=\unitlength,page=1]{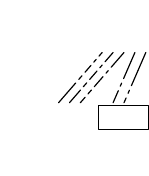}}%
    \put(0.79053309,0.34040493){\makebox(0,0)[lt]{\lineheight{1.25}\smash{\begin{tabular}[t]{l}$\beta$\end{tabular}}}}%
    \put(0,0){\includegraphics[width=\unitlength,page=2]{ReduceOneStrand1.pdf}}%
  \end{picture}%
\endgroup%

\hspace{3mm} 
% second leftmost: 
%% Creator: Inkscape 1.1.2 (b8e25be833, 2022-02-05), www.inkscape.org
%% Accompanies image file 'ReduceOneStrand2.pdf' (pdf, eps, ps)
\begingroup%
  \makeatletter%
  \providecommand\color[2][]{%
    \errmessage{(Inkscape) Color is used for the text in Inkscape, but the package 'color.sty' is not loaded}%
    \renewcommand\color[2][]{}%
  }%
  \providecommand\transparent[1]{%
    \errmessage{(Inkscape) Transparency is used (non-zero) for the text in Inkscape, but the package 'transparent.sty' is not loaded}%
    \renewcommand\transparent[1]{}%
  }%
  \providecommand\rotatebox[2]{#2}%
  \newcommand*\fsize{\dimexpr\f@size pt\relax}%
  \newcommand*\lineheight[1]{\fontsize{\fsize}{#1\fsize}\selectfont}%
  \ifx\svgwidth\undefined%
    \setlength{\unitlength}{71.37316954bp}%
    \ifx\svgscale\undefined%
      \relax%
    \else%
      \setlength{\unitlength}{\unitlength * \real{\svgscale}}%
    \fi%
  \else%
    \setlength{\unitlength}{\svgwidth}%
  \fi%
  \global\let\svgwidth\undefined%
  \global\let\svgscale\undefined%
  \makeatother%
  \begin{picture}(1,1.18112865)%
    \lineheight{1}%
    \setlength\tabcolsep{0pt}%
    \put(0,0){\includegraphics[width=\unitlength,page=1]{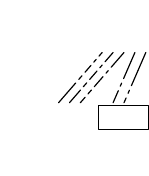}}%
    \put(0.79053309,0.34040493){\makebox(0,0)[lt]{\lineheight{1.25}\smash{\begin{tabular}[t]{l}$\beta$\end{tabular}}}}%
    \put(0,0){\includegraphics[width=\unitlength,page=2]{ReduceOneStrand2.pdf}}%
  \end{picture}%
\endgroup%

\hspace{3mm} 
% second rightmost: 
%% Creator: Inkscape 1.1.2 (b8e25be833, 2022-02-05), www.inkscape.org
%% Accompanies image file 'ReduceOneStrand3.pdf' (pdf, eps, ps)
\begingroup%
  \makeatletter%
  \providecommand\color[2][]{%
    \errmessage{(Inkscape) Color is used for the text in Inkscape, but the package 'color.sty' is not loaded}%
    \renewcommand\color[2][]{}%
  }%
  \providecommand\transparent[1]{%
    \errmessage{(Inkscape) Transparency is used (non-zero) for the text in Inkscape, but the package 'transparent.sty' is not loaded}%
    \renewcommand\transparent[1]{}%
  }%
  \providecommand\rotatebox[2]{#2}%
  \newcommand*\fsize{\dimexpr\f@size pt\relax}%
  \newcommand*\lineheight[1]{\fontsize{\fsize}{#1\fsize}\selectfont}%
  \ifx\svgwidth\undefined%
    \setlength{\unitlength}{71.37316954bp}%
    \ifx\svgscale\undefined%
      \relax%
    \else%
      \setlength{\unitlength}{\unitlength * \real{\svgscale}}%
    \fi%
  \else%
    \setlength{\unitlength}{\svgwidth}%
  \fi%
  \global\let\svgwidth\undefined%
  \global\let\svgscale\undefined%
  \makeatother%
  \begin{picture}(1,1.18112865)%
    \lineheight{1}%
    \setlength\tabcolsep{0pt}%
    \put(0,0){\includegraphics[width=\unitlength,page=1]{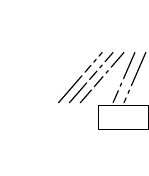}}%
    \put(0.79053309,0.34040493){\makebox(0,0)[lt]{\lineheight{1.25}\smash{\begin{tabular}[t]{l}$\beta'$\end{tabular}}}}%
    \put(0,0){\includegraphics[width=\unitlength,page=2]{ReduceOneStrand3.pdf}}%
  \end{picture}%
\endgroup%

\hspace{3mm} 
% rightmost: 
%% Creator: Inkscape 1.1.2 (b8e25be833, 2022-02-05), www.inkscape.org
%% Accompanies image file 'ReduceOneStrand4.pdf' (pdf, eps, ps)
\begingroup%
  \makeatletter%
  \providecommand\color[2][]{%
    \errmessage{(Inkscape) Color is used for the text in Inkscape, but the package 'color.sty' is not loaded}%
    \renewcommand\color[2][]{}%
  }%
  \providecommand\transparent[1]{%
    \errmessage{(Inkscape) Transparency is used (non-zero) for the text in Inkscape, but the package 'transparent.sty' is not loaded}%
    \renewcommand\transparent[1]{}%
  }%
  \providecommand\rotatebox[2]{#2}%
  \newcommand*\fsize{\dimexpr\f@size pt\relax}%
  \newcommand*\lineheight[1]{\fontsize{\fsize}{#1\fsize}\selectfont}%
  \ifx\svgwidth\undefined%
    \setlength{\unitlength}{71.37316954bp}%
    \ifx\svgscale\undefined%
      \relax%
    \else%
      \setlength{\unitlength}{\unitlength * \real{\svgscale}}%
    \fi%
  \else%
    \setlength{\unitlength}{\svgwidth}%
  \fi%
  \global\let\svgwidth\undefined%
  \global\let\svgscale\undefined%
  \makeatother%
  \begin{picture}(1,1.18112865)%
    \lineheight{1}%
    \setlength\tabcolsep{0pt}%
    \put(0,0){\includegraphics[width=\unitlength,page=1]{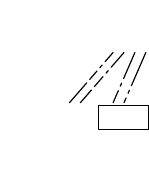}}%
    \put(0.79053309,0.34040493){\makebox(0,0)[lt]{\lineheight{1.25}\smash{\begin{tabular}[t]{l}$\beta'$\end{tabular}}}}%
    \put(0,0){\includegraphics[width=\unitlength,page=2]{ReduceOneStrand4.pdf}}%
  \end{picture}%
\endgroup%

\caption{This series of pictures (from left to right) illustrates how one strand from the closure of the leftmost braid can be reduced. By abuse of notation, the symbols $\beta$ and $\beta'$ in this figure represent the sub-braids with the rightmost $d_k$ strands of the braids $\beta$ and $\beta'$ respectively.
\label{Fig:ReduceOneStrand}}
\end{figure}  

We push this over strand up and shrink it to reduce one strand from the braid $B$, as shown in the third drawing of Figure~\ref{Fig:ReduceOneStrand}. After that, this strand becomes an over strand between the $(1)$-th and the $(d_1)$-st strand yielding the braid
$$(d_k, s_k)(d_{k-1}, s_{k-1}) \dots (d_2, s_2)(d_1, s_1-1)(\sigma_{1}\dots \sigma_{d_1-1})\beta',$$
with $\beta' =  \sigma_{i_1-1}\dots \sigma_{i_a-1}.$

Since the first $s_1+\dots + s_k-1$ strands of the last braid are still overcrossing strands with the last crossings $\sigma_{i}$ with $i\geq d_1$, we can push the sub-braid $(\sigma_{1}\dots \sigma_{d_1-1})$ up to be on the top of the braid so that we obtain the braid
$$(\sigma_{s_1+\dots + s_k}\dots \sigma_{s_1+\dots + s_k+d_1-2})(d_k, s_k)(d_{k-1}, s_{k-1}) \dots (d_2, s_2)(d_1, s_1-1)\beta'.$$

Finally we push the sub-braid $(\sigma_{s_1+\dots + s_k}\dots \sigma_{s_1+\dots + s_k+d_1-2})$ anticlockwise around the braid closure to place it below the sub-braid $\beta'$ as illustrated in the last drawing of Figure~\ref{Fig:ReduceOneStrand}. After that, we obtain the braid
$$B''= (d_k, s_k)\dots (d_2, s_2)(d_1, s_1-1)\beta'(\sigma_{s_1+\dots + s_k}\dots \sigma_{s_1+\dots + s_k+d_1-2}). \qedhere$$
\end{proof} 

Intuitively speaking, Proposition~\ref{Proposition1} is ``removing'' the leftmost overcrossing strand in the Lorenz braid by ambient-isotoping its neighbouring undercrossing strand to a horizontal strand under the sub-braid $\beta$.

\begin{lemma}\label{lemma1}
Let $\vec{d}_{L'} = \langle d_1^{s_1}\rangle$ be a Lorenz vector with $s_1\geq1$ and $\beta =  \sigma_{i_1}\dots \sigma_{i_a}$ be a braid with at least\ $s_1+ d_1$ strands and no crossings $\sigma_{i_j}$ with $i_j\leq s_1$. Then, there is an ambient isotopy that takes the closure of the braid $$B'=(d_1, s_1)\beta$$ to the closure of the braid 
$$B''= \beta'(\sigma_{1}\dots \sigma_{d_1-1})^{s_1}$$
with $\beta' =  \sigma_{i_1-s_1}\dots \sigma_{i_a-s_1}$.
\end{lemma} 

\begin{proof} 
After applying $s_1-1$ times Proposition~\ref{Proposition1} to $B'$, we obtain the following braid, called $B_{s_1-1}$,
$$(d_1, 1)\beta_{s_1-1}(\sigma_{2}\dots \sigma_{d_1})^{s_1-1}$$ with $\beta_{s_1-1} =  \sigma_{i_1-(s_1-1)}\dots \sigma_{i_a-(s_1-1)}.$

The sub-braid $\beta_{s_1-1}(\sigma_{2}\dots \sigma_{d_1})^{s_1-1}$ has no crossing $\sigma_{i_j}$ with $i_j = 1$. Hence we can  trivially destabilize $B_{s_1-1}$ on its left side to yield the braid
$$(\sigma_{1}\dots \sigma_{d_1-1})\beta_{s_1}(\sigma_{1}\dots \sigma_{d_1-1})^{s_1-1}$$
with $\beta_{s_1} =  \sigma_{i_1-s_1}\dots \sigma_{i_a-s_1}.$ 
Then, we push the sub-braid $(\sigma_{1}\dots \sigma_{d_1-1})$ anticlockwise one time around the braid closure to obtain the braid $$\beta_{s_1}(\sigma_{1}\dots \sigma_{d_1-1})^{s_1}.$$
So, we set $\beta' = \beta_{s_1}$. 
\end{proof} 

\begin{proposition}\label{proposition2}
Let $\vec{d}_{L'} = \langle d_1^{s_1}, \dots, d_k^{s_k}\rangle$ be a Lorenz vector and $\beta =  \sigma_{i_1}\dots \sigma_{i_a}$ be a braid with at least $s_1+\dots + s_k + d_k$ strands and  no crossings $\sigma_{i_j}$ with $i_j\leq s_1+\dots + s_k$. Then, there is an ambient isotopy that takes the closure of the braid $$B'=(d_k, s_k)(d_{k-1}, s_{k-1}) \dots (d_1, s_1)\beta$$ to the closure of the braid 
$$(\sigma_{1}\dots \sigma_{d_1-1})^{s_1}(\sigma_{1}\dots \sigma_{d_2-1})^{s_2}\dots (\sigma_{1}\dots \sigma_{d_k-1})^{s_k}\beta'.$$
with $\beta' =  \sigma_{i_1-(s_1+\dots + s_k)}\dots \sigma_{i_a-(s_1+\dots + s_k)}$.
\end{proposition} 

\begin{proof}
If $k = 1$, then $B'$ is equivalent to the braid $B''= \beta'(\sigma_{1}\dots \sigma_{d_1-1})^{s_1}$
with $\beta' =  \sigma_{i_1-s_1}\dots \sigma_{i_a-s_1}$ by Lemma~\ref{lemma1}. We push the sub-braid $(\sigma_{1}\dots \sigma_{d_1-1})^{s_1}$ clockwise one time around the braid closure to obtain the braid $$(\sigma_{1}\dots \sigma_{d_1-1})^{s_1}\beta'.$$

Consider then that $k>1$. We apply $s_1+\dots + s_{k-1}$ times Proposition~\ref{Proposition1} to $B'$ to obtain the following braid
\begin{align*}
& B_{s_1+\dots + s_{k-1}} = (d_k, s_k)\beta_{s_1+\dots + s_{k-1}}(\sigma_{s_k+1}\dots \sigma_{s_k+d_1-1})^{s_1} \\
& (\sigma_{s_k+1}\dots \sigma_{s_k+d_2-1})^{s_2}\dots (\sigma_{s_k+1}\dots \sigma_{s_k+d_{k-1}-1})^{s_{k-1}},
\end{align*}
with $\beta_{s_1+\dots + s_{k-1}} =  \sigma_{i_1-(s_1+\dots + s_{k-1})}\dots \sigma_{i_a-(s_1+\dots + s_{k-1})}.$
Then, we apply Lemma~\ref{lemma1} to $B_{s_1+\dots + s_{k-1}}$ to obtain the following braid 
$$\beta_{s_1+ \dots + s_k}(\sigma_{1}\dots \sigma_{d_1-1})^{s_1}(\sigma_{1}\dots \sigma_{d_2-1})^{s_2}\dots (\sigma_{1}\dots \sigma_{d_k-1})^{s_k},$$
with $\beta_{s_1+ \dots + s_k} =  \sigma_{i_1-(s_1+ \dots + s_k)}\dots \sigma_{i_a-(s_1+ \dots + s_k)}.$ 
Finally, we push the sub-braid $(\sigma_{1}\dots \sigma_{d_1-1})^{s_1}\dots (\sigma_{1}\dots \sigma_{d_k-1})^{s_k}$ clockwise one time around the braid closure to obtain the braid 
$$(\sigma_{1}\dots \sigma_{d_1-1})^{s_1}(\sigma_{1}\dots \sigma_{d_2-1})^{s_2}\dots (\sigma_{1}\dots \sigma_{d_k-1})^{s_k}\beta_{s_1+ \dots + s_k}. \qedhere $$ 
\end{proof}

For every integer $n$ we have a family of links defined as follows: 

\begin{definition}
Let $r_1, \dots, r_k, r_{k+1}, s_1, \dots, s_{k}, d$ be non-negative integers.
The $T^n$-link $$T^n((r_1,s_1), \dots, (r_k,s_k), (r_{k+1}; d))$$ is defined to be the closure of the following braid
\[(\sigma_{d+1}\sigma_{d+2}\dots\sigma_{d+r_1-1})^{s_1}\dots(\sigma_{d+1}\sigma_{d+2}\dots\sigma_{d+r_k-1})^{s_k}\Delta^{n}_{r_{k+1}, d+r_{k+1}}\]
if $k+1>1$, $2\leq r_1< \dots < r_k\leq r_{k+1}$, all $s_i>0$, and $d\geq 0$, or 
\[\Delta^{n}_{r_{k+1}, d+r_{k+1}}\]
if $k+1 = 1, r_{1}\geq 1$ and $d\geq 0$, where $T^n(1; d)$ is the unlink with $d+1$ components given by the trivial braid with $d+1$ strands.
\end{definition}

If \( n = 0 \), then \( T^0 \)-links are the same as T-links, up to a number of unknots. This occurs because Birman and Kofman in \cite{Birman-Kofman:NewTwistOnLorenzLinks} did not consider splittable links that have trivial components parallel to the left boundary component of the Lorenz template. This is harmless when one studies splittable links componentwise. However, in the present paper we keep track of these components because
our goal is to represent all links in \(S^3\). 
 
\begin{theorem}\label{Theorem1}
Lorenz-like links in the Lorenz-like template \(\mathscr{L}(0,n)\) are equivalent to \(T^n\)-links. In other words, every link that can be embedded into the Lorenz-like template \(\mathscr{L}(0,n)\) is a 
\(T^n\)-link, and every \(T^n\)-link can be embedded into the Lorenz-like template \(\mathscr{L}(0,n)\).
\end{theorem}  

\begin{proof}
Consider $L$ a link in the Lorenz-like template $\mathscr{L}(0,n)$. Let $L'$ be the Lorenz link associated with the link $L$.

If $L$ has a trip number equal to zero in the Lorenz-like template $\mathscr{L}(0,n)$, then $L'$ is a collection of trivial loops that are parallel to the two boundary components of the Lorenz template $\mathscr{L}(0,0)$. 
Since $L$ is obtained from $L'$ by applying $n$ half twists on the strands in the right strip of the Lorenz template $\mathscr{L}(0,0)$,
it follows that $L$ is a collection of trivial loops with $n$ half twists applied to some of them. Hence, $L$ is the $T^n$-link $T^n(r_1; d)$ with $r_1\geq 1$ and $d\geq 0$.

If the Lorenz braid of $L'$ has exactly one overcrossing strand (or equivalently, the Lorenz link $L'$ has trip number equal to one) and exactly one strand in the right strip of the template $\mathscr{L}(0,0)$, then $L$ and $L'$ are the same unlink $T^n(1; d)$ with  $d\geq 0$.

Now consider the case where there is at least one overcrossing strand and at least two strands in the right strip of the template $\mathscr{L}(0,0)$ in the Lorenz braid of $L'$. The link $L$ can be an obviously splittable link with trivial loops that are parallel to the left boundary component of the template $\mathscr{L}(0,n)$. If the link $L$ has some of these trivial loops, we ignore them for now.
Let $\vec{d}_{L'} = \langle d_1^{s_1}, \dots, d_k^{s_k}\rangle$ be the Lorenz vector associated with $L$, where \( d_1, \dots, d_k, d, s_1, \dots, s_k \) are natural numbers such that \( 2 \leq d_1 < \dots < d_k \) and \( s_i > 0 \). Furthermore, there exists a natural number $d\geq 0$ such that $L$ is the closure of the following braid with $p = s_1+\dots + s_k + d_k+d$ strands
$$(d_k, s_k)(d_{k-1}, s_{k-1}) \dots (d_2, s_2)(d_1, s_1)\Delta^{n}_{d_k+d, p}.$$
By Proposition~\ref{proposition2}, there is an ambient isotopy that takes the closure of the last braid to the closure of the braid 
$$(\sigma_{1}\dots \sigma_{d_1-1})^{s_1}\dots (\sigma_{1}\dots \sigma_{d_k-1})^{s_k}\Delta^{n}_{d_k+d, d_k+d},$$
which represents the $T^n$-link $T^n((d_1,s_1), \dots, (d_k,s_k), (d_k+d; 0))$.

If we now consider that $L$ has trivial loops that are parallel to the left boundary component of the template $\mathscr{L}(0,n)$, then there exists a number $d'\geq 0$ such that $L$ is the closure of the braid
$$(\sigma_{d'+1}\dots \sigma_{d'+d_1-1})^{s_1}\dots (\sigma_{d'+1}\dots \sigma_{d'+d_k-1})^{s_k}\Delta^{n}_{d_k+d, d_k+d+d'},$$
which represents the $T^n$-link $T^n((d_1,s_1), \dots, (d_k,s_k), (d_k+d; d'))$.
 
Similarly, if we apply the reverse procedure, we can see that every $T^n$-link can be embedded into the Lorenz-like template $\mathscr{L}(0,n)$. 
\end{proof} 

In particular, we obtain \cite[Theorem 1]{Birman-Kofman:NewTwistOnLorenzLinks} as a specific case of Theorem~\ref{Theorem1} and simplify its proof.

%\begin{corollary}\label{associated}
%%%%Consider \( d, d', d_1, \dots, d_k, s_1, \dots, s_k \) non-negative %%%integers such that \( 2 \leq d_1 < \dots < d_k \) and \( s_i > 0 \).
%%%The T-link $T((d_1,s_1), \dots, (d_k,s_k))$ is a Lorenz link %%associated with the $T^n$-link $T^n((d_1,s_1), \dots, (d_k,s_k), (d_k+d; d'))$.
%%\end{corollary} 

%\begin{proof} 
%It follows from the proof of Theorem~\ref{Theorem1}
%\end{proof}

\begin{corollary}\label{corollaryL(m, 0)}
Lorenz-like links in the Lorenz-like template $\mathscr{L}(m,0)$ are equivalent to $T^m$-links. In other words, every link that can be embedded into the Lorenz-like template $\mathscr{L}(m,0)$  is a $T^m$-link and every $T^m$-link can be embedded into the Lorenz-like template $\mathscr{L}(m,0)$ . 
\end{corollary}  

\begin{proof} 
It follows from Theorem~\ref{Theorem1} since the templates $\mathscr{L}(m,0)$ and $\mathscr{L}(0,m)$ have the same family of links.
\end{proof}  

\begin{theorem}\label{Thm:InfiniteTripNumber}
Every T-link has infinitely many representations in each template $\mathscr{L}(0,2n)$  with $n<0$. 
\end{theorem}

\begin{proof} 
Consider the T-link $L= T((r_1,s_1), \dots, (r_k,s_k))$. By definition, the parameters of $L$ satisfy $2\leq r_1< \dots < r_k$, $s_1, \dots, s_{k}>0$. However, we can consider that $s_k>1$ otherwise the T-link representation of $L$ can be destabilized on its right to yield the T-link $T((r_1,s_1), (r_2,s_2), \dots, (r_{k-1},s_{k-1}+1)).$ So we assume that $s_{k}>1$.

The T-link $L$ is given by the braid
\[(\sigma_{1}\dots\sigma_{r_1-1})^{s_1}(\sigma_{1}\dots\sigma_{r_2-1})^{s_2}\dots(\sigma_{1}\dots\sigma_{r_k-1})^{s_k},\]
which is equivalent to the braid
\[(\sigma_{1}\dots\sigma_{r_1-1})^{s_1}\dots(\sigma_{1}\dots\sigma_{r_k-1})^{s_k-1}(\sigma_{1}\dots\sigma_{r'_k-1})^{1}\]
or 
\[(\sigma_{1}\dots\sigma_{r_1-1})^{s_1}\dots(\sigma_{1}\dots\sigma_{r_k-1})^{s_k-1}(\sigma_{1}\dots\sigma_{r'_k-1})^{1-nr'_k}(\sigma_{1}\dots\sigma_{r'_k-1})^{nr'_k},\]
which is the same as
\[(\sigma_{1}\dots\sigma_{r_1-1})^{s_1}\dots(\sigma_{1}\dots\sigma_{r_k-1})^{s_k-1}(\sigma_{1}\dots\sigma_{r'_k-1})^{1-nr'_k}\Delta^{2n}_{r'_{k}, r'_{k}}\]
for any $r'_k\geq r_k$ and $n<0$.
The last braid represents the $T^{2n}$-link $$L'=T^{2n}((r_1,s_1), \dots, (r_k,s_k-1), (r'_k,1-nr'_k), (r'_{k}; 0)).$$

Hence the T-link \(L\) is equivalent to the \(T^{2n}\)-link \(L'\). By
Theorem~\ref{Theorem1}, the link \(L'\) is represented in the
Lorenz-like template \(\mathscr L(0,2n)\). Since \(r'_k\) can be chosen
arbitrarily large, this gives infinitely many representations of \(L\) in
the template \(\mathscr L(0,2n)\).
\end{proof} 

\begin{definition}\label{generalizedT-link}
Let $r_1, \dots, r_k, r_{k+1}, s_1, \dots, s_{k}, s_{k+1}, d$ be  integers such that $2\leq r_1< \dots < r_k < r_{k+1}$, $s_1, \dots, s_{k}>0$, $s_{k+1}\in \mathbb{Z}-\lbrace -1, 0, 1 \rbrace$, and $d\geq 0$.
The generalised T-link $$T((r_1,s_1), \dots, (r_k,s_k), (r_{k+1}, s_{k+1}), d)$$ is defined to be the closure of the following braid
\[ (\sigma_{d+1}\dots\sigma_{d+r_1-1})^{s_1}\dots(\sigma_{d+1}\dots\sigma_{d+r_k-1})^{s_k}(\sigma_{d+1}\dots\sigma_{d+r_{k+1}-1})^{s_{k+1}}.\]
\end{definition}
 
If $s_{k+1}<0$, then $$(\sigma_{d+1}\dots\sigma_{d+r_{k+1}-2}\sigma_{d+r_{k+1}-1})^{s_{k+1}} = (\sigma_{d+r_{k+1}-1}^{-1}\sigma_{d+r_{k+1}-2}^{-1}\dots\sigma_{d+1}^{-1})^{-s_{k+1}}$$ in the braid group. So, more concretely, in case $s_{k+1}<0$, then the generalised T-link $T((r_1,s_1), \dots, (r_k,s_k), (r_{k+1}, s_{k+1}), d)$ is the closure of the braid
\[ (\sigma_{d+1}\dots\sigma_{d+r_1-1})^{s_1}\dots(\sigma_{d+1}\dots\sigma_{d+r_k-1})^{s_k}(\sigma_{d+r_{k+1}-1}^{-1}\dots\sigma_{d+1}^{-1})^{-s_{k+1}}.\]

The exclusion of the cases \(s_{k+1}=\pm1\) is a normalization convention.
Indeed, if the final twisting parameter is equal to \(1\) or \(-1\), then
the corresponding braid closure admits a destabilization at the end of the
last twisting block. After applying this destabilization, one obtains an
equivalent generalised T-link representative in which either the final
twisting block has been absorbed into the preceding braid word, or the
final twisting parameter satisfies \(|s_{k+1}|>1\). Thus, up to
destabilization, no generality is lost by assuming
\(s_{k+1}\in\mathbb Z\setminus\{-1,0,1\}\). 

\begin{theorem}\label{Theorem101}
Generalised T-links $T((r_1,s_1), \dots, (r_k,s_k), (r_{k+1}, s_{k+1}), d)$ with $s_{k+1}\geqslant nr_{k+1}$ and $T^{2n}$-links are in equivalence if $n\leq 0$. If $n>0$, then generalised T-links $T((r_1,s_1), \dots, (r_k,s_k), (r_{k+1}, s_{k+1}), d)$ with $s_{k+1}\geqslant nr_{k+1}$ together with the trivial links $T((2,1), d)$ are equivalent to  $T^{2n}$-links.
\end{theorem} 

\begin{proof} 
The trivial links of the form $T((2,1), d)$ are $T^{2n}$-links of the form $T^{2n}(1; d)$. Now let $L$ be a generalised T-link $L = T((r_1,s_1), \dots, (r_{k+1}, s_{k+1}), d)$ with $s_{k+1}\geqslant nr_{k+1}$.
Hence $s_{k+1}-nr_{k+1}\geq 0$. The link $L$ is given by the braid 
\[ (\sigma_{d+1}\dots\sigma_{d+r_1-1})^{s_1}\dots(\sigma_{d+1}\dots\sigma_{d+r_k-1})^{s_k}(\sigma_{d+1}\dots\sigma_{d+r_{k+1}-1})^{s_{k+1}},\]
%if $s_{k+1}>0$ or
%\[ (\sigma_{1+d}\dots\sigma_{d+r_1-1})^{s_1}\dots(\sigma_{1+d}\dots\sigma_{d+r_k-1})^{s_k}(\sigma_{d+r_{k+1}-1}^{-1}\sigma_{d+r_{k+1}-2}^{-1}\dots\sigma_{1+d}^{-1})^{-s_{k+1}},,\]
which is equivalent to the braid
\begin{align*}
& (\sigma_{d+1}\dots\sigma_{d+r_1-1})^{s_1}\dots(\sigma_{d+1}\dots\sigma_{d+r_k-1})^{s_k} \\
& (\sigma_{d+1}\dots\sigma_{d+r_{k+1}-1})^{s_{k+1}-nr_{k+1}}\Delta^{2n}_{r_{k+1}, d+r_{k+1}},
\end{align*}
which represents the $T^{2n}$-link $$T^{2n}((r_1,s_1), \dots, (r_k,s_k), (r_{k+1}, s_{k+1}-nr_{k+1}), (r_{k+1}; d))$$ 
(it becomes $T^{2n}((r_1,s_1), \dots, (r_k,s_k), (r_{k+1}; d))$ in case $s_{k+1}-nr_{k+1}=0$).

Now let $L$ be a $T^{2n}$-link $$T^{2n}((r_1,s_1), \dots, (r_k,s_k), (r_{k+1}; d))$$ given by the braid
\[(\sigma_{d+1}\sigma_{d+2}\dots\sigma_{d+r_1-1})^{s_1}\dots(\sigma_{d+1}\sigma_{d+2}\dots\sigma_{d+r_k-1})^{s_k}\Delta^{2n}_{r_{k+1}, d+r_{k+1}},\]
which is the same as the braid
\[(\sigma_{d+1}\dots\sigma_{d+r_1-1})^{s_1}\dots(\sigma_{d+1}\dots\sigma_{d+r_k-1})^{s_k}(\sigma_{d+1}\sigma_{d+2}\dots\sigma_{d+r_{k+1}-1})^{nr_{k+1}}.\]
Consider $r_{k} = r_{k+1}$. Then, this braid represents the generalised T-link $$T((r_1,s_1), \dots, (r_{k-1},s_{k-1}), (r_k,s_k + nr_k), d)$$
if $s_k + nr_k\neq 0$, with $s_k + nr_k> nr_k$,
otherwise $$T((r_1,s_1), \dots, (r_{k-1},s_{k-1}), d),$$
with $s_{k-1}>nr_{k-1}$ since $n$ is negative in this case.
If $r_{k} \neq r_{k+1}$, the last braid represents the generalised T-link $$T((r_1,s_1), \dots, (r_k,s_k), (r_{k+1}, nr_{k+1}), d)$$
with $nr_{k+1}= nr_{k+1}.$

Finally, consider a $T^{2n}$-link of the form $T^{2n}(r_{1}; d)$. The link $L$ is given by the braid
\[\Delta^{2n}_{r_{1}, d+r_{1}}\]
with $r_{1}\geq 1$ and $d\geq 0$, where $T^{2n}(1; d)$ is the unlink with $d+1$ components. If $r_{1}=1$, then $L$ is the generalised T-link $T((2,1), d)$, where we have $1> 2n$ if $n\leq 0$. If $r_{1}>1$, then $L$ is the generalised T-link 
$$T((r_{1}, nr_{1}), d),$$ 
with $nr_{1} = nr_{1}$.

Therefore, if $n\leq 0$, then $T^{2n}$-links and generalised T-links $$T((r_1,s_1), \dots, (r_{k+1}, s_{k+1}), d)$$ with $s_{k+1}\geqslant nr_{k+1}$ are equivalent. If $n>0$, then generalised T-links $T((r_1,s_1), \dots, (r_k,s_k), (r_{k+1}, s_{k+1}), d)$ with $s_{k+1}\geqslant nr_{k+1}$ together with the trivial links $T((2,1), d)$ are equivalent to  $T^{2n}$-links.
\end{proof} 

\begin{theorem}\label{parameterization}
Generalised T-links $T((r_1,s_1), \dots, (r_k,s_k), (r_{k+1}, s_{k+1}), d)$ with $s_{k+1}\geqslant vr_{k+1}$  are in equivalence
with Lorenz-like links in the Lorenz-like template $\mathscr{L}(0,2v)$ if $v\leq 0$.
\end{theorem} 

\begin{proof}
It follows from Theorems~\ref{Theorem1} and \ref{Theorem101}.
\end{proof}  

\begin{corollary}\label{corollary1002}
Let $L$ be the generalised T-link $$T((r_1,s_1), \dots, (r_k,s_k), (r_{k+1}, s_{k+1}), d)$$ with $s_{k+1}\geqslant nr_{k+1}$. Then, $L$ is the $T^{2n}$-link $$T^{2n}((r_1,s_1), \dots, (r_k,s_k), (r_{k+1}, s_{k+1}-nr_{k+1}), (r_{k+1}; d)).$$
\end{corollary} 

\begin{proof}
It follows from the proof of Theorem~\ref{Theorem101}.
\end{proof}   

\begin{theorem}\label{theorem1003}
Fix any negative integer $n$. Then, the set of all generalised T-links $$T((r_1,s_1), \dots, (r_k,s_k), (r_{k+1}, s_{k+1}), d)$$ with $s_{k+1}\geqslant nr_{k+1}$ contains all links in $\mathbb{S}^3$.
\end{theorem} 

\begin{proof} 
By Theorem~\ref{parameterization}, if we consider that $s_{k+1}\geqslant nr_{k+1}$, then the generalised T-links $T((r_1,s_1), \dots, (r_k,s_k), (r_{k+1}, s_{k+1}), d)$ form the set of all links in the  Lorenz-like template $\mathscr{L}(0,2n)$.
The Lorenz-like template $\mathscr{L}(0,2n)$ contains all links in $\mathbb{S}^3$ by Proposition~\ref{Prop:UniversalTemplates}. Therefore,  the set of all generalised T-links $T((r_1,s_1), \dots, (r_k,s_k), (r_{k+1}, s_{k+1}), d)$ with $s_{k+1}\geqslant nr_{k+1}$ contains all links in $\mathbb{S}^3$.
\end{proof} 

\begin{corollary}\label{corollary1003'}
The set of all generalised T-links is the set of all links in $\mathbb{S}^3$. 
\end{corollary} 

\begin{proof} 
This follows from Theorem~\ref{theorem1003} because, in particular, we can choose a negative integer $n$ and consider the stricter condition $s_{k+1}\geqslant nr_{k+1}$.
\end{proof}

\section{Final remarks and further directions}
\label{Sec:FinalRemarks}

Theorem~\ref{Thm:AllLinksGeneralisedTLinks} is a representation result,
not a classification theorem. It shows that every link in \(S^3\) admits
at least one representative as a generalised T-link, but it does not
address when two different sets of parameters determine isotopic links.
Thus, a natural next problem is to understand the equivalence relation on
generalised T-link parameters induced by ambient isotopy.

This question is nontrivial for the classical families contained
in the class of generalised T-links. T-links, Lorenz links, and twisted
torus links often have different braid descriptions representing the same link. In the universal setting, this non-uniqueness becomes even more
pronounced: a fixed link (up to ambient isotopy) can be contained in a Lorenz-like template in many
different ways. The generalised T-link description gives a
braid-theoretic framework in which part of this redundancy can be studied.

Another natural direction is to study classical invariants in terms of
the generalised T-link parameters. Since the family contains all links but
retains a structured braid form, one may ask how invariants such as braid
index, genus, Alexander polynomial, knot Floer homology, hyperbolic
volume, or the JSJ decomposition behave with respect to the parameters.
One may also ask to what extent known results for T-links, Lorenz links,
and twisted torus links extend to this universal setting.

Finally, the parametrisation suggests the possibility of defining new
quantities from generalised T-link representatives. Such quantities may not be link invariants as the representation is not
unique. However, by taking suitable minima or maxima over all
generalised T-link representatives of a fixed link, one may obtain new
invariants or computable refinements of existing ones.

\bibliographystyle{amsplain}  
\bibliography{biblio_MathSciNet.bib} 
\end{document}